\documentclass[amstex, dvopdfmx, 11pt]{amsart}
\usepackage{graphicx}
\usepackage{amssymb}
\usepackage{mathrsfs}

\newtheorem{Pro}{Proposition}[section]
\newtheorem{lemma}{Lemma}[section]
\newtheorem{Cor}{Corollary}[section]
\newtheorem{Mythm}{Theorem}

\theoremstyle{definition}

\newtheorem{Ex}{Example}[section]
\theoremstyle{definition}
\newtheorem{Rem}{Remark}[section]
\numberwithin{equation}{section}

\begin{document}
% Definition of title page:
\title[Quasiconformal equivalence]{On the quasiconformal equivalence of
    Dynamical Cantor sets}
\author{
    Hiroshige Shiga    % insert author(s) here
}
\address{Department of Mathematics,
Kyoto Sangyo University \\
Motoyama, Kamigamo, Kita-ku Kyoto, Japan} 

\email{shiga@cc.kyoto-su.ac.jp}
\date{\today}    % optional
\keywords{Cantor sets, Quasiconformal maps}
\subjclass[2010]{Primary 30C62, Secondary 30F25.}
\thanks{
The author was partially supported by the Ministry of Education, Science, Sports
and Culture, Japan;
Grant-in-Aid for Scientific Research (B), 16H03933, 2016--2020.}

\begin{abstract}
The complement of a Cantor set in the complex plane is itself regarded as a Riemann surface of infinite type.
The problem of this paper  is the quasiconformal equivalence of such Riemann surfaces.
Particularly, we are interested in Riemann surfaces given by Cantor sets which are created through dynamical methods.
We discuss the quasiconformal equivalence for the complements of Cantor Julia sets of rational functions and random Cantor sets.
We also consider the Teichm\"uller distance between random Cantor sets.

\end{abstract}
\maketitle
\section{Introduction}
Let $E$ be a Cantor set in the Riemann sphere $\widehat{\mathbb C}$, that is, a totally disconnected perfect set in $\widehat{\mathbb C}$.
The standard middle one-third Cantor set $\mathcal{C}$ is a typical example.
We consider the complement $X_{E}:=\widehat{\mathbb C}\setminus E$ of the Cantor set $E$. 
It is an open Riemann surface with uncountable boundary components.
We are interested in the quasiconformal equivalence of such Riemann surfaces.
In the previous paper \cite{ShigaPrepri}, we show that the complement of the limit set of a Schottky group is quasiconformally equivalent to $X_{\mathcal C}$, the complement of $\mathcal C$ (\cite{ShigaPrepri} Theorem 6.2).
In this paper, we discuss the quasiconformal equivalence for the complements of Cantor Julia sets of hyperbolic rational functions and random Cantor sets (see \S2 for the terminologies). 
We establish the following theorems.
\begin{Mythm}
\label{Thm:Main}
	Let $f$ be a rational function of degree $d>1$ and $\mathcal J$ be the Julia set of $f$.
	Suppose that $f$ is hyperbolic and $\mathcal J$ is a Cantor set.
	Then, the complement $X_{\mathcal J}$ of $\mathcal J$ is quasiconformally equivalent to $X_{\mathcal C}$.
\end{Mythm}
As for random Cantor sets, we obtain the followings.
%\begin{Mythm}
%\label{Thm:Main2}
%	Let $E(\omega)$ be a random Cantor set for $\omega=(q_n)_{n=1}^{\infty}$.
%	Suppose that $\omega$ has lower and upper bounds.
%		Then, $X_{E(\omega)}$ is quasiconformally equivalent to $X_{\mathcal C}$.
%	Conversely, if  $X_{E(\omega)}$ is quasiconformally equivalent to $X_{\mathcal C}$, then $\omega$ has an upper bound.	
%\end{Mythm}

\begin{Mythm}
	\label{Thm:Main2}
	Let $\omega=(q_n)_{n=1}^{\infty}$ and $\widetilde{\omega}=(\tilde q_n)_{n=1}^{\infty}$ be sequences with $\delta$-lower bound.
	We put
	\begin{equation}
		d(\omega, \widetilde{\omega})= \sup_{n\in \mathbb N} \max  \left\{\left | \log \frac{1-\tilde q_n}{1-q_n}\right |, |\tilde q_n -q_n|\right\}.
	\end{equation}
	\begin{enumerate}
		\item If $d(\omega, \widetilde{\omega})<\infty$, then there exists an $\exp(C(\delta)d(\omega, \widetilde{\omega}))$-quasiconformal mapping $\varphi$ on $\widehat{\mathbb{C}}$ such that $\varphi (E(\omega))=E(\widetilde{\omega})$, where $C(\delta)>0$ is a  constant depending only on $\delta$;
	\item if $\lim_{n\to \infty}\log \frac{1-\tilde q_n}{1-q_n}=0$, then $E(\widetilde\omega)$ is asymptotically conformal to $E(\omega)$, that is, there exists a quasiconformal mapping $\varphi$ on $\widehat{\mathbb C}$ with $\varphi (E(\omega))=E(\widetilde\omega)$ such that for any $\varepsilon>0$, $\varphi |_{U_{\varepsilon}}$ is $(1+\varepsilon)$-quasiconformal on a neighborhood $U_{\varepsilon}$ of $E(\omega)$. 
	\end{enumerate}
	
\end{Mythm}

From above results and a result \cite{ShigaPrepri} Theorem 6.2, immediately we obtain;
\begin{Cor}
\label{Cor1}
	Let $E$ be a Cantor set which is a Julia set of  a rational function satisfying the conditions in Theorem \ref{Thm:Main}.
	Then, the complement of the limit set of a Schottky group $G$ is quasiconformally equivalent to $X_{E}$.
\end{Cor}

As consequences of Theorem \ref{Thm:Main2} (1), we obtain;
\begin{Cor}
\label{Cor3}
		Let $E(\omega)$ be a random Cantor set for $\omega=(q_n)_{n=1}^{\infty}$.
	Suppose that $\omega$ has lower and upper bounds.
		Then, $X_{E(\omega)}$ is quasiconformally equivalent to $X_{\mathcal C}$.	
\end{Cor}

We have also the following.
\begin{Cor}
\label{Cor2}
Let $E$ be a Cantor set as in Corollaries 1.1 or 1.2.
Then, the Cantor set $E$ is quasiconformally removable, that is, every quasiconformal mapping on the complement of $E$ is extended to a quasiconformal mapping on the Riemann sphere.	
\end{Cor}

It is known (\cite{Sario-Nakai} V. 11F. Theorem) that the random Cantor set $E(\omega)$ for $\omega$ is of capacity zero if and only if
\begin{eqnarray}
\label{eqn:Cap=0}
	\prod_{n=1}^{\infty}(1-q_n)^{2^{-n}}=0.
\end{eqnarray}
Hence if $\{q_n\}_{n=1}^{\infty}$ rapidly converges to one as it satisfies (\ref{eqn:Cap=0}), then $X_{E(\omega)}$ is not quasiconformally equivalent to $X_{\mathcal C}$ because the positivity of the capacity of closed sets in the plane is preserved by quasiconformal mappings (cf. \cite{Sario-Nakai} III. Theorem 8 H).
In fact, we can say more:
\begin{Mythm}
\label{Them:Main3}
	If $\omega$ does not have an upper bound, then $X_{E(\omega)}$ is not quasiconformally equivalent to $X_{\mathcal C}$.
\end{Mythm}

The proof of Theorem \ref{Thm:Main2} gives the following.
\begin{Cor}
\label{Cor:Hausdorff}
Let $\omega$ and $\widetilde\omega$ be sequences satisfying the same conditions as in Theorem \ref{Thm:Main2} (2).
Then, the Hausdorff dimension of $E(\widetilde\omega)$ is the same as that of $E(\omega)$.
\end{Cor}

\noindent
{\bf Acknowledgement.}
The author thanks Prof. K. Matsuzaki for his valuable comments.
This research was partly done during the author's stay in the Erwin Schr\"odinger institute at Vienna. He also thanks the institute for its brilliant support to his research.

\section{Preliminaries}
\subsection{Complex dynamics.}
We begin with a small and brief introduction of complex dynamics.
We may refer textbooks on the topic, e.~ g.~ \cite{Carleson-Gamelin} for a general theory of complex dynamics.

Let $f$ be a rational function of degree $d>1$ on $\mathbb C$.
We denote by $f^n$ the $n$ times iterations of $f$.
The Fatou set $\mathcal F$ of $f$ is the set of points in $\widehat{\mathbb C}$ which have neighborhoods where $\{f^n\}_{n=1}^{\infty}$ is a normal family.
The complement of $\mathcal F$, which is denoted by $\mathcal J$, is called the Julia set of $f$.

A rational function $f$ is called \emph{hyperbolic} if it is expanding near $\mathcal J$.
More precisely, if $\mathcal J \not\ni\infty$, then $f$ is hyperbolic if there exist a constant $A>1$ and a smooth metric $\sigma (z)|dz|$ in a neighborhood $U$ of $\mathcal J$ such that
\begin{equation*}
	\sigma (f(z))|f' (z)|\geq A\sigma (z), \quad z\in \mathcal J
\end{equation*}
(see \cite{Carleson-Gamelin} V. 2).
If $\infty\in \mathcal J$, the hyperbolicity of $f$ is defined by conjugation of M\"obius transformations as usual. 

The hyperbolicity is also characterized in terms of the Euclidean metric and the dynamical behavior of rational functions as well.
\begin{Pro}[\cite{Carleson-Gamelin} V. 2. Lemma 2.1 and Theorem 2.2]
\label{Pro:hyperbolic}
	A rational function $f$ is hyperbolic if and only if every critical point belongs to $\mathcal F$ and is attracted to an attracting cycle.
	If $\infty\not\in \mathcal J$, then the hyperbolicity of $f$ is equivalent to the existence of $m\geq 1$ such that $|(f^m)'|>1$ on $\mathcal J$.
\end{Pro}

\subsection{Random Cantor sets} (cf. \cite{Sario-Nakai} I. 6).
\label{sebsec.Random Cantor sets}
Let $\omega=(q_n)_{n=1}^{\infty}=(q_1, q_2, \dots )$ be a sequence of real numbers with $0<q_n<1$ for each $n\in \mathbb N$.
We construct a Cantor set $E(\omega)$ for $\omega$ inductively as follows.

First, we remove an open interval $J_1$ of length $q_1$ from $E_0:=I=[0, 1]$ so that $I\setminus J_1$ consists of two closed intervals $I_1^{1}, I_1^{2}$ of the same length.
We put $E_1=\cup_{i=1}^{2}I_1^{i}$.
We remove an open interval of length $|I_1^{i}|q_2$ from each $I_1^{i}$ so that the remainder $E_2$ consists of four closed intervals of the same length, where $|J|$ is the length of an interval $J$.
Inductively, we define $E_{k+1}$ from $E_k=\cup_{i=1}^{2^k}I_k^{i}$ by removing an open interval of length $|I_k^{i}|q_{k+1}$ from each closed interval $I_k^{i}$ of $E_k$ so that $E_{k+1}$ consists of $2^{k+1}$ closed intervals of the same length.
The random Cantor set $E(\omega)$ for $\omega$ is defined by
\begin{equation*}
	E(\omega)=\cap_{k=1}^{\infty}E_k.
\end{equation*}
It is a generalization of the standard middle one-third Cantor set $\mathcal C$.
In fact, $\mathcal C=E(\omega_0)$ for $\omega_0=(\frac{1}{3})_{n=1}^{\infty}=(\frac{1}{3}, \frac{1}{3}, \dots )$.

We say that a sequence $\omega=(q_n)_{n=1}^{\infty}$ as above is of ($\delta$-)lower bound if there exists a $\delta>0$ such that $q_n\geq\delta$ for any $n\in \mathbb N$.
We also say that a sequence $\omega$ has a ($\delta$-)upper bound if $q_n\leq 1-\delta$ for any $n\in\mathbb N$.

\subsection{Hausdorff dimension.}
Let $E$ be a subset of $\mathbb C$ and $\alpha >0$.
We consider a countable open covering $\{U_i\}_{i\in \mathbb N}$ of $E$
with $\textrm{diam} (U_i)<r$ for a given $r>0$.
Then, we set
\begin{equation*}
	\Lambda_{\alpha}^{r}(E):=\inf\left\{\sum_{i\in \mathbb N}(\textrm{diam} (U_i) )^{\alpha} \right\},
\end{equation*}
where the infimum is taken over all countable open covering $\{U_i\}_{i\in \mathbb N}$ with $\textrm{diam}(U_i)<r$.
We put
\begin{equation*}
	\Lambda_{\alpha}(E)=\lim_{r\to 0}\Lambda_{\alpha}^{r}(E)
\end{equation*}
and the Hausdorff dimension $\textrm{dim}_{H}(E)$ of $E$ by
\begin{equation*}
	\textrm{dim}_{H}(E)=\inf \{ \alpha \mid \Lambda_{\alpha}(E)=0 \}.
\end{equation*}

\subsection{The quasiconformal equivalence of open Riemann surfaces.}
We say that two Riemann surfaces $R_1, R_2$ are quasiconformally equivalent if there exists a quasiconformal homeomorphism between them. 
We also say that they are quasiconformally equivalent near the ideal boundary if there exist compact subset $K_j$ of $R_j$ $(j=1, 2)$ and quasiconformal homeomorphism $\varphi$ from $R_1\setminus K_1$ onto $R_2\setminus K_2$.

It is obvious that if $R_1, R_2$ are quasiconformally equivalent, then they are quasiconformally equivalent near the ideal boundary.
On the other hand, we have shown that the converse is not true in general.
In fact, we have constructed two Riemann surfaces which are not quasiconformally equivalent while they are homeomorphic to each other and quasiconformally equivalent near the ideal boundary (\cite{ShigaPrepri} Example 3.1). 
We also give a sufficient condition for Riemann surfaces to be quasiconformally equivalent (\cite{ShigaPrepri} Theorem 5.1).
\begin{Pro}
\label{Pro:qc-equiv}
Let $R_1, R_2$ be open Riemann surfaces which are homeomorphic to each other and quasiconformally  equivalent near the ideal boundary. 
If the genus of $R_1$ is finite, then $R_1$ and $R_2$ are quasiconformally equivalent.	
\end{Pro}

\section{Proof of Theorem \ref{Thm:Main}}
Let $f$ be a hyperbolic rational function with the Cantor Julia set $\mathcal J$.
We show that $X_{\mathcal J}$ is quasiconformally equivalent to $X_{\mathcal C}$.
By Proposition \ref{Pro:qc-equiv}, it suffices to show that there exists a compact subset $K$ of $\mathcal F$ such that $\mathcal F\setminus K$ is quasiconformally equivalent to the complement of a compact subset of $X_{\mathcal C}$.

Considering the conjugation by M\"obius transformations,
we may assume that $\mathcal J$ does not contain $\infty$.
Since $\mathcal J$ is a Cantor set, the Fatou set $\mathcal F$ is connected.
Furthermore, it follows from Proposition \ref{Pro:hyperbolic} that $\mathcal F$ contains an attracting fixed point $z_0$ of $f$.
Then, we may find a simply connected neighborhood $\Omega_0$ of $z_0$ such that $f(\overline{\Omega}_0)\subset \Omega_0$.
We may take $\Omega_0$ so that the boundary $\partial\Omega_0$ is a smooth Jordan curve and it does not contain the forward orbits of critical points of $f$.

For each $k\in \mathbb N$, let $\Omega_k$ be a connected component of $f^{-k}(\Omega_0)$ containing $z_0$.
We may assume that $\Omega_1$ is bounded by at least two Jordan curves.
Then, each $\Omega_k$ is bounded by mutually disjoint finitely many smooth Jordan curves and 
we have
\begin{equation*}
	z_0\in \Omega_0\subset \overline{\Omega}_0\subset \Omega_1\subset \overline{\Omega_1}\dots \subset \overline{\Omega}_k\subset \Omega_{k+1}\subset \dots 
\end{equation*}
and
\begin{equation*}
	\mathcal F=\cup_{k=0}^{\infty}\Omega_k.
\end{equation*}

Since $f$ is hyperbolic, the Julia set $\mathcal J$ does not contain critical points.
Moreover, there exists a simply connected neighborhood $V_{z}$ of each $z\in \mathcal J$ such that $f|_{V_z}$ is injective on $V_z$ (Proposition \ref{Pro:hyperbolic}).
Hence, from compactness of $\mathcal J$ there exist disks $V_1, \dots , V_n$ for some $n\in \mathbb N$ such that $\mathcal J\subset \cup_{j=1}^{n}V_j$ and $f|_{V_j}$ is injective $(1\leq j\leq n)$.
Then, we show;
\begin{lemma}
	There exists $k_0\in \mathbb N$ such that for any $k\geq k_0$,  each connected component of $\Omega_{k+1}\setminus \overline{\Omega}_k$ is contained in some $V_j$ $(1\leq j\leq n)$.
\end{lemma}
\begin{proof}
	Since $f(\Omega_{k+1}\setminus\overline{\Omega_k})=\Omega_{k}\setminus\overline{\Omega_{k-1}}$ and $\Omega_{k+1}\supset\Omega_k$, we see that if every connected component of $\Omega_{k_0+1}\setminus\overline{\Omega_{k_0}}$ is contained in some $V_j$, then so is for $k\geq k_0$.
	Hence, we may find such a $k_0$ to show the statement of the lemma.
	
	Suppose that for any $k\in \mathbb N$, there exists a connected component $\Delta_k$ of $\Omega_{k+1}\setminus\Omega_k$ such that $\Delta_k$ is not contained in any $V_j$ $(j=1, 2, \dots , n)$.
	Thus, for sufficiently large $k$, $\Delta_k$ is contained in $\cup_{j=1}^{n}V_j$ but it is not contained in any $V_j$.
	By taking a subsequence if necessary, we may assume that $\Delta_k\cap V_j\not=\emptyset$ and $\Delta_k\cap V_{j'}\not=\emptyset$ for some $j, j' \in\{1, 2, \dots , n\}$.
	Let $x$ be an accumulation point of $\{\Delta_k\}_{k=1}^{\infty}$.
	Then, $x$ has to be in $\mathcal J$ because $\mathcal F=\cup_{k=1}^{\infty}\Omega_k$.
	
	The Julia set $\mathcal J$ is totally disconnected.
	Hence, if a sequence $\{x_{k_m}\}_{m=1}^{\infty}$ $(x_{k_m}\in \Delta_{k_m})$ converges to $x$, then $\{\Delta_{k_m}\}_{m=1}^{\infty}$ also converges to $\{x\}$.
	In other words, for any neighborhood $U$ of $x$, there exists $m_0\in \mathbb N$ such that for any $m\geq m_0$, $\Delta_{k_m}$ is contained in $U$.
	However, $x\in \mathcal J$ is in some $V_{j_0}$ because $\mathcal J\subset \cup_{j=1}^{n}V_j$.
	Therefore, $\Delta_{k_m}$ is contained in $V_{j_0}$ if $m$ is sufficiently large.
	Thus, we have a contradiction.	
\end{proof}

We take $k_0\in \mathbb N$ in the above lemma.
Let $\omega_1, \omega_2, \dots , \omega_{\ell}$ be the set of connected components of $\Omega_{k_0+1}\setminus\overline{\Omega_{k_0}}$.
Each $\omega_j$ is bounded by a finite number, say $L(j)+1$, of mutually disjoint simple closed curves.
We may assume that $L(j)>1$ $(j=1, 2, \dots \ell)$.
Note that the number of boundary components of $\partial\Omega_{k_0}\cap\partial\Omega_{k_0+1}$ is equal to $\ell$.
It is because $\partial (\Omega_{k_0+1}\setminus\overline{\Omega_{k_0}})$ consists of mutually disjoint simple closed curves in $\widehat{\mathbb C}$, and $\overline{\Omega_{k_0}}$ is compact.

\begin{figure}
	\centering
	\includegraphics[width=15cm]{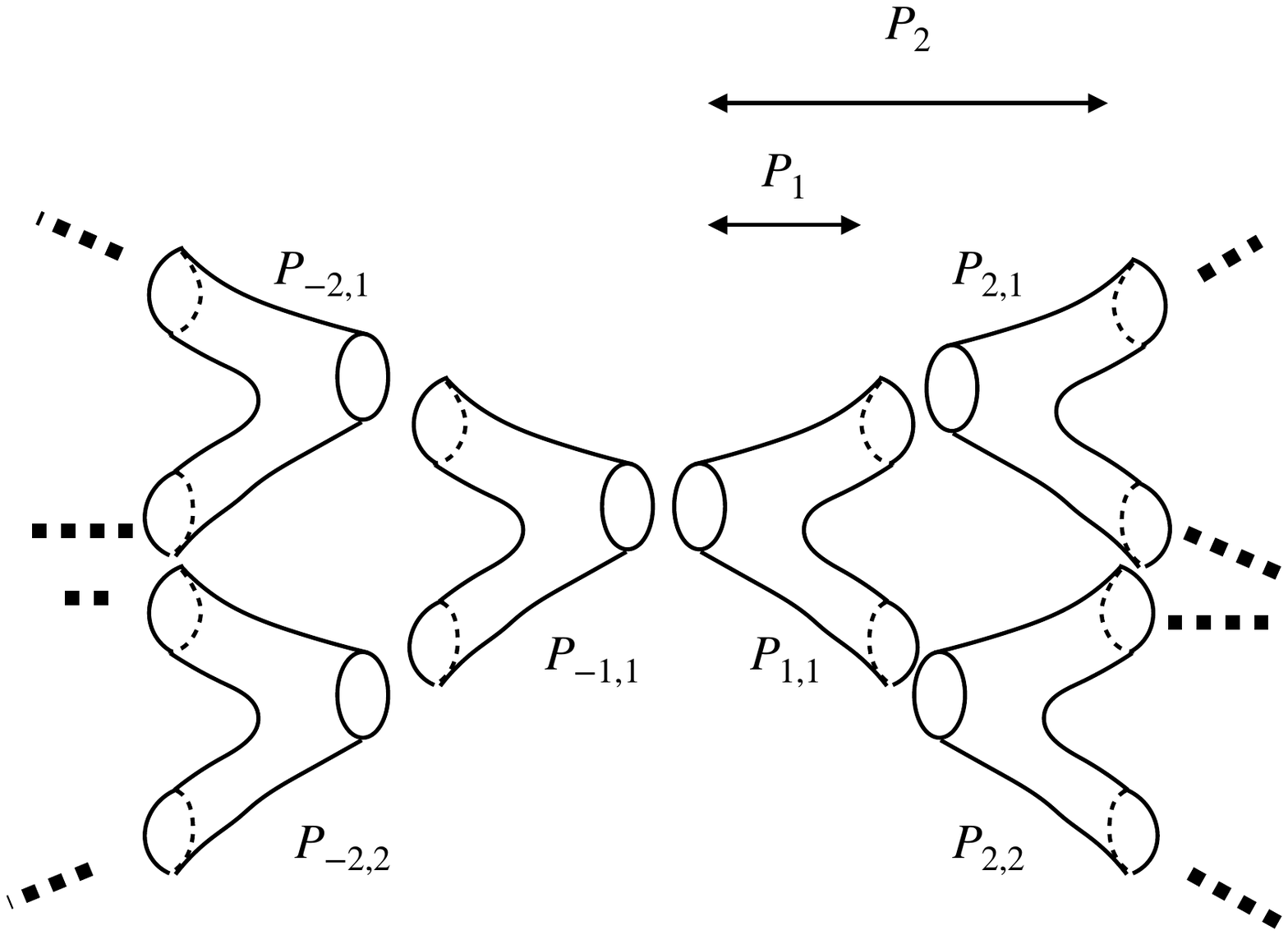}
	\caption{The middle one-third Cantor set}
	\label{Fig.Sstandard}
\end{figure}

For any $k>k_0$ and for a connected component $\Delta$ of $\Omega_{k+1}\setminus\overline{\Omega_k}$, we have $f^{k-k_0}(\Delta)\subset \Omega_{k_0+1}\setminus\overline{\Omega_{k_0}}$ and $f^{k-k_0}$ is conformal in $\Delta$ since $\Delta$ is contained in some $V_j$.
Hence, $\Delta$ is conformally equivalent to $\omega_{J}$ for some $J\in\{1, 2, \dots , \ell\}$.
Therefore, if $k>k_0$, then $\Omega_{k+1}\setminus\overline{\Omega_{k}}$ contains at most $\ell$ conformally different connected components.

Now, we consider the middle one-third Cantor set $\mathcal C$ and $X_{\mathcal C}:=\widehat{\mathbb C}\setminus\mathcal C$.
It is not hard to see that $X_{\mathcal C}$ admits a pants decomposition $\{P_{i,j}\}_{i\in \mathbb Z\setminus\{0\}, j\in \{1, \dots , 2^{|i|-1}\}}$ as in Figure \ref{Fig.Sstandard}.
In fact, we may take all $P_{i, j}$ so that they are conformally equivalent to each other.
Let $P_{N}$ $(N\in \mathbb N)$ be a subdomain of $X_{\mathcal C}$ consisting of $P_{i, j}$ for $i=1, \dots , N$ and $j=1, \dots , 2^{i-1}$.
We see that $P_N$ is bounded by $2^N+1$ mutually disjoint simple closed curves. 

Let $N_0\in \mathbb N$ be the largest number with $2^{N_0}+1\leq\ell$.
We put 
$$
K:=\overline{P_{N_0}\cup_{j=1}^{\ell_0}P_{N_0+1, j}},
$$
where $\ell_0=\ell -2^{N_0}-1$.
Then, $K$ is a compact subset of $X_{\mathcal C}$ bounded by $\ell$ simple closed curves. 
We denote by $C_1, \dots , C_{\ell}$, where $C_1\subset\partial P_{1, 1}$.
We may take a subdomain $G_1$ of $X_{\mathcal C}$ so that $G_1\setminus K$ is quasiconformally equivalent to $\Omega_{k_0+1}\setminus\overline{\Omega_{k_0}}$ as follows.

We take the largest number $L_1$ with $2^{L_1}\leq L(1)$.
Then, 
\begin{equation*}
	\overline{G_{1, 1}}:=\overline{\left (\cup_{i=1}^{L_1}\cup_{j=1, \dots , 2^{|i|}}P_{-i, j}\right )\cup \left (\cup_{j=1, \dots , L(1)-2^{L_1}}P_{-L_1-1, j}\right )}
\end{equation*}
is a closed subdomain of $X_{\mathcal C}$ with $L(1)+1$ boundary curves.
Hence, $G_{1, 1}$ is quasiconformally equivalent to $\omega_1$ since both of them are planar domains bounded by the same number of closed curves.

Similarly, we may construct subdomains $G_{1, 2}, \dots , G_{1, \ell}$ such that $\partial G_{1, j}\cap\partial K=C_j$ and each $G_{1, j}$ is quasiconformally equivalent to $\omega_{j}$ $(j=1, 2, \dots , \ell)$. 
Combining $K$ with $G_{1, 1}, \dots , G_{1, \ell}$, we obtain a desired subdomain $G_1$.

By using the same argument as above, we have a subdomain $G_2$ of $X_{\mathcal C}$ such that $G_1\subset G_2$ and $G_{2}\setminus\overline{G_1}$ is quasiconformally equivalent to $\Omega_{k_0+2}\setminus\overline{\Omega_{k_0+1}}$.
Moreover, we may use this argument inductively and we obtain a exhaustion $\{G_i\}_{i=1}^{\infty}$ of $X_{\mathcal C}$ such that
\begin{equation*}
	K\subset G_1 \subset G_2 \subset \dots \subset G_{i}\subset G_{i+1}\subset \dots , \quad X_{\mathcal C}=\cup_{i=1}^{\infty}G_i,
\end{equation*}
and 
$G_{i+1}\setminus \overline{G_{i}}$ are quasiconformally equivalent to $\Omega_{k_0+i+1}\setminus\overline{\Omega_{k_0+i}}$.

Now, we note the following.
\begin{Pro}
\label{Pro:Gluing}
	Let $R_1, R_2$ be Riemann surfaces. 
	We consider simple closed curves $\alpha_i$ in $R_i$ with $R_i\setminus{\alpha_i}=S^{(i)}_1\cup S^{(i)}_2$, where $S^{(i)}_1$ and $S^{(i)}_2$ are mutually disjoint subsurface of $R_i$ $(i=1, 2)$.
	Suppose that there exist quasiconformal mappings $f_{j} : S^{(1)}_{j}\to S^{(2)}_{j}$ $(j=1, 2)$ such that $f_1(\alpha_1)=f_2(\alpha_1)=\alpha_2$.
	Then, there exists a quasiconformal mapping $f : R_1\to R_2$. Moreover, the maximal dilatation of $f$ depends only on those of $f_1, f_2$ and the local behavior of those mappings near $\alpha_1$.
\end{Pro}
We may apply this proposition to domains $G_{i+1}\setminus\overline{G_i}$ and $\Omega_{k_0+i+1}\setminus\overline{\Omega_{k_0+i}}$ $(i=1, 2, \dots )$.
Noting that there only finitely many conformal equivalence classes in those domains, we verify that $X_{\mathcal C}\setminus K=\cup_{i\in \mathbb N}(G_{i+1}\setminus\overline{G_i})$ and $\mathcal F\setminus \overline{\Omega_{k_0+1}}=\cup_{i\in \mathbb N}(\Omega_{k_0+i+1}\setminus\overline{\Omega_{k_0+i}})$ are quasiconformally equivalent.
\qed

\section{Proof of Theorem \ref{Thm:Main2}}
{\it Proof of (1).}
We divide the proof into several steps.

\medskip
\noindent
{\bf Step 1 : Analyzing random Cantor sets.}
Let $\omega=(q_n)_{n=1}^{\infty}$ and $\widetilde{\omega}=(\tilde{q}_n)_{n=1}^{\infty}$ be sequences with $\delta$-lower bound.
We take $E_{k}=\cup_{i=1}^{2^k}I_k^{i}$ and $\widetilde E_{k}=\cup_{i=1}^{2^k}\widetilde I_k^{i}$ as in \S \ref{sebsec.Random Cantor sets} for $\omega$ and $\widetilde\omega$, respectively.
In fact, $I_{k}^{i}$ (resp. $\widetilde I_{k}^i$) is located at the left of $I_k^{i+1}$ (resp. $\widetilde I_{k}^{i+1}$) for $i=1, 2, \dots , 2^k-1$.
The set $[0, 1]\setminus E_k$ (resp. $[0, 1]\setminus \widetilde E_k$) consists of $2^k-1$ open intervals $J_k^1, \dots , J_k^{2^k-1}$ (resp. $\widetilde J_k^1, \dots , \widetilde J_k^{2^k-1}$).
Each $J_k^i$ (resp. $\widetilde J_k^i$) is located between $I_k^i$ and $I_k^{i+1}$ (resp. $\widetilde I_k^i$ and $\widetilde I_k^{i+1}$).
 
Because of the construction, we have
\begin{equation*}
	|I_{k+1}^i|=\frac{1}{2}(1-q_k)|I_k^i|.
\end{equation*}
Therefore, we have
\begin{equation}
	|I_k^i|=2^{-k}\prod_{j=1}^{k}(1-q_j).
\end{equation}
Next, we estimate the length of $J_k^{i}$.

In construction $E_{k+1}$ from $E_k$, we obtain open intervals $I_{k+1}^{2i-1}$, $I_{k+1}^{2i}$ and the closed interval $J_{k+1}^{2i-1}$ such that $I_k^{i}=I_{k+1}^{2i-1}\cup J_{k+1}^{2i-1}\cup I_{k+1}^{2i}$ for each $i, k$ (Figure \ref{Fig.Interval}).

\begin{figure}
	\centering
	\includegraphics[width=14cm]{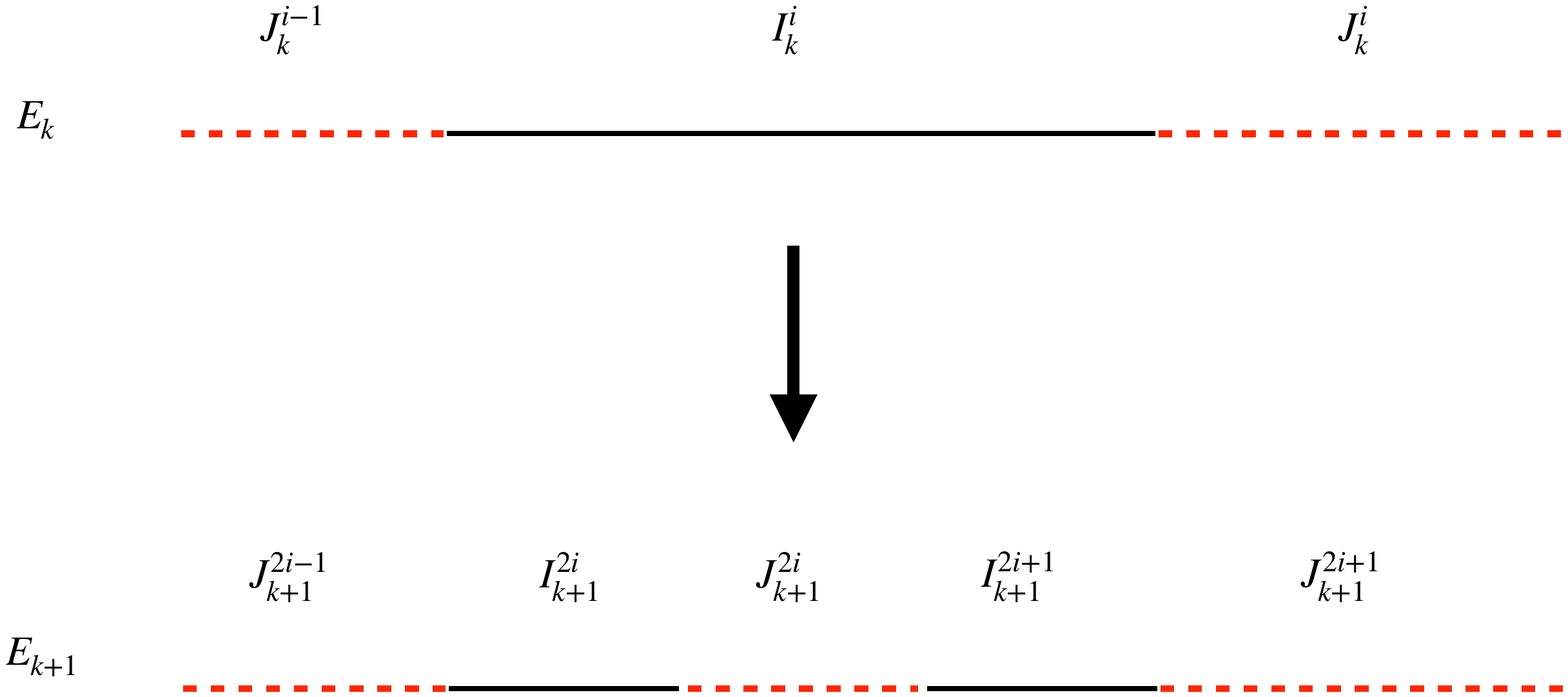}
	\caption{}
	\label{Fig.Interval}
\end{figure}

If $i$ is odd, we have
\begin{equation}
\label{eqn:J-length1}
	|J_{k+1}^{i}|=|I_{k}^{i}|q_{k+1}=\frac{2q_{k+1}}{1-q_{k+1}} |I_{k+1}^1|\geq 2\delta |I_{k+1}^1|,
\end{equation}
as $q_{k+1}\geq \delta$.

If $i$ is even, then $i=2^{\ell}m$ for an integer $\ell$ with $1\leq \ell\leq k$ and an odd number $m$.
Since $J_{k+1}^{i}$ is located between $I_{k+1}^{i}$ and $I_{k+1}^{i+1}$, we see that $J_{k+1}^{i}=J_{k}^{i/2}=J_{k}^{2^{\ell -1}m}$.
Repeating this argument, we have
$J_{k+1}^{i}=J_{k-\ell+1}^{m}$.
Since $m$ is odd, we conclude from (\ref{eqn:J-length1}) that
\begin{eqnarray}
\label{eqn:J-even1}
	|J_{k+1}^{i}|&=&|J_{k-\ell+1}^{m}|=2^{-k+\ell}q_{k-\ell +1}\prod_{j=1}^{k-\ell}(1-q_j) \nonumber  \\ &\geq & 2^{-k+1}\delta \prod_{j=1}^{k}(1-q_j) 
	\geq 4\delta |I_{k+1}^1|
\end{eqnarray}
as $q_{k-\ell +1}\geq\delta$.

Thus, we obtain the following from (\ref{eqn:J-length1}) and (\ref{eqn:J-even1}).
\begin{lemma}
\label{Lemma}
	Let $I_k^i$ and $J_{k+1}^i$ be the same ones as above for a sequence $\omega=(q_n)_{n=1}^{\infty}$ with $\delta$-lower bound.
	Then,
	\begin{equation}
	\label{eqn:J}
		|J_{k+1}^{i}|\geq 2\delta |I_{k+1}^1|
	\end{equation}
	hold for $i=1, 2, \dots , 2^{k+1}-1$.
\end{lemma}

\medskip
\noindent
{\bf Step 2: Constructing a pants decomposition.}
We draw a circle $C_k^i$ centered at the midpoint of $I_{k}^i$ with radius $\frac{1}{2}(1+\delta)|I_{k}^1|$ for each $k\in \mathbb N$ and $1\leq i\leq 2^k$.
From (\ref{eqn:J}), we see that $C_k^i\cap C_k^j=\emptyset$ if $i\not=j$.
Since
\begin{equation*}
	\frac{1}{2}\cdot \delta|I_{k+1}^1|<\frac{1}{2}\cdot\delta|I_{k}^1|,
\end{equation*}
we also see that $C_{k+1}^i\cap C_k^j=\emptyset$.
Therefore, $\cup_{k=1}^{\infty}\cup_{i=1}^{2^k}C_k^i$ gives a pants decomposition for $X_{E(\omega)}$.

We draw circles $\widetilde{C}_k^i$ for $\widetilde\omega$ by the same way. Then, we also see that $\cup_{k=1}^{\infty}\cup_{i=1}^{2^k}\widetilde C_k^i$ gives a pants decomposition for $X_{E(\widetilde\omega)}$.

\medskip
\noindent
{\bf Step 3: Analyzing a pair of pants.} We denote by $P_k^i$ a pair of pants bounded by $C_k^i, C_{k+1}^{2i-1}$ and $C_{k+1}^{2i}$.
We consider the complex structure of $P_k^i$ so that we may assume that the center of $C_k^i$ is the origin with radius $\frac{1}{2}(1+\delta)|I_{k}^1|$.
Then, the centers of $C_{k+1}^{2i-1}$ and $C_{k+1}^{2i}$ are
\begin{equation*}
	-\frac{1}{2}q_{k+1}|I_{k}^1|-\frac{1}{4}\left (1+\delta\right )(1-q_{k+1})|I_{k}^1|
\end{equation*}
and
\begin{equation*}
	\frac{1}{2}q_{k+1}|I_{k}^1|+\frac{1}{4}\left (1+\delta\right )(1-q_{k+1})|I_{k}^1|,
\end{equation*}
respectively.

By applying an affine map $z\mapsto \alpha z+\beta$ for some $\alpha >0, \beta\in \mathbb R$ to $P_k^i$ so that the circle $C_k^i$ is sent a circle centered at the origin with radius $1+\delta$. 
We denote the circle by $C_{k,1}$. 
Then, the circle $C_{k+1}^{2i-1}$ is sent a circle $C_{k, 2}$ centered at
\begin{equation*}
	-x_k :=-q_{k+1}-\frac{1}{2}\left (1+\delta\right )(1-q_{k+1}) =-\frac{1}{2}\left\{\left (1+\delta\right )+\left (1-\delta\right )q_{k+1}\right\}
\end{equation*}
with radius
\begin{equation*}
	r_k :=\frac{1}{2}\left (1+\delta\right )(1-q_{k+1})
\end{equation*}
and $C_{k+1}^{2i}$ is sent a circle $C_{k, 3}$ centered at $x_k$ with radius $r_k$.
We may conformally identify $P_k^i$ with a pair of pants $\mathcal P_k$
bounded by $C_{k, 1}, C_{k, 2}$ and $C_{k, 3}$.

Similarly, we consider a pair of pants $\widetilde{P}_k^i$ bounded by $\widetilde C_k^i, \widetilde  C_{k+1}^{2i-1}$ and $\widetilde C_{k+1}^{2i}$, and apply an affine map to the pair of pants $\widetilde P_{k}^i$ so that the circle $\widetilde C_k^i$ is mapped a circle centered at the origin with radius $1+\delta$, which is the same circle as the image of $C_k^i$ above.
We denote by $\widetilde{C}_{k, i}$ the image of $\widetilde{C}_k^i$ $(i=1, 2, 3)$.  
We may conformally identify $\widetilde{P}_k^i$ with a pair of pants $\widetilde{\mathcal P}_k$ bounded by $\widetilde C_{k, 1}, \widetilde C_{k, 2}$ and $\widetilde C_{k, 3}$, where 
$\widetilde C_{k,1}$ is the same circle as $C_{k,1}$, $\widetilde C_{k,2}$ is centered at
\begin{equation*}
	-\widetilde x_k :=-\frac{1}{2}\left\{\left (1+\delta\right )+\left (1-\delta\right )\widetilde q_{k+1}\right\}
\end{equation*}
wirh radius
\begin{equation*}
	\widetilde r_k :=\frac{1}{2}\left (1+\delta\right )(1-\widetilde q_{k+1})
\end{equation*}
and $\widetilde C_{k, 3}$ is centered at $\widetilde{x}_k$ with radius $\widetilde{r}_k$.

\medskip
\noindent
{\bf Step 4 : Constructing intermediate pairs of pants.}
By applying $z\mapsto (x_k/\widetilde x_k)z$ to $\widetilde{\mathcal P}_k$, we obtain a pair of pants $\widehat{P}_k$.
The pair of pants $\widehat{P}_k$ is bounded by $\widehat C_{k, 1}, \widehat C_{k, 2}$ and $\widehat C_{k, 3}$. Each $\widehat C_{k, i}$ is corresponding to $\widetilde{C}_{k, i}$ $(i=1, 2, 3)$.
Note that for each $i$, the center of $\widehat C_{k, i}$ is $x_k$, the same as that of $C_{k, i}$, and $\widehat{P}_k$ is conformally equivalent to $\widetilde{\mathcal P}_k$.
The radius of $\widehat{C}_{k, 1}$ is
\begin{equation*}
	\left (1+\delta\right )\cdot\frac{x_k}{\widetilde x_k}=\left (1+\delta\right )\frac{\left (1+\delta\right )+\left (1-\delta\right )q_{k+1}}{\left (1+\delta\right )+\left (1-\delta\right )\widetilde q_{k+1}},
\end{equation*}
and the radius of $\widehat C_{k, 2}, \widehat C_{k, 3}$ is
\begin{equation*}
	\widehat r_k :=\frac{1}{2}\left (1+\delta\right )(1-\widetilde q_{k+1})\left (1+\delta\right )\frac{\left (1+\delta\right )+\left (1-\delta\right )q_{k+1}}{\left (1+\delta\right )+\left (1-\delta\right )\widetilde q_{k+1}}.
\end{equation*}

Now, we take an intermediate pair of pants $P_k'$ bounded by $\widehat C_{k, 1}, C_{k, 2}$ and $C_{k, 3}$.

\medskip
\noindent
{\bf Step 5 : Making quasiconformal mappings, I.}
In the following the argument, we use a notation $d(\varphi)$ for a quasiconformal mapping $\varphi$ as
\begin{equation*}
	d(\varphi)=\log K(\varphi),
\end{equation*}
where $K(\varphi)$ is the maximal dilatation of $\varphi$.

We suppose that $q_{k+1}\geq\widetilde{q}_{k+1}$. Then, we have
\begin{equation*}
	\widehat{r}_k\geq r_k =\frac{1}{2}\left (1+\delta\right )(1-q_{k+1}).
\end{equation*}
In other words, the radius of $\widehat C_{k, 2}$, $\widehat C_{k, 3}$ is not smaller than that of $C_{k, 2}$, $C_{k, 3}$.

Let $C_{k, +}$ be a circle centered at $x_k$ with radius
\begin{equation*}
	\widetilde{R}_k :=\left (1+\delta\right )\frac{x_k}{\widetilde x_k}-x_k,
\end{equation*}
so that $C_{k, +}$ is tangent with $\widehat C_{k, 1}$.

We consider two circular annuli $A_{k, +}$ bounded by $C_{k, +}$ and $\widehat C_{k, 3}$, $A_{k, +}'$ bounded by $C_{k, +}$ and $C_{k, 3}$.
Here, we use the following well-known fact.
\begin{lemma}
\label{lemma:modulus}
	For annuli $A_i =\{0<r_i<|z|<R_i<\infty\}$ ($i=1, 2$), there exists a quasiconformal mapping $\varphi : A_1\to A_2$ such that
	\begin{eqnarray*}
		\varphi (r_1e^{i\theta})=r_2e^{i\theta} \\
		\varphi (R_1E^{i\theta})=R_2e^{i\theta}
	\end{eqnarray*}
	and
	\begin{equation*}
		K(\varphi)=e^{d(\varphi)},
	\end{equation*}
	where 
	\begin{equation*}
		d(\varphi)=\left |\log \frac{\log R_1 -\log r_1}{\log R_2-\log r_2}\right |.
	\end{equation*}
\end{lemma}
 
It follows from Lemma \ref{lemma:modulus} that there exists a quasiconformal mapping $\varphi_{k, +} : A_{k, +}\to A_{k, +}'$ such that
\begin{equation*}
	d(\varphi_{k, +})=\log \frac{\log \widetilde{R}_k - \log  r_k}{\log \widetilde R_k -\log \widehat r_k},
\end{equation*}
\begin{equation}
\label{eqn:bdybehavior}
	\varphi_{k, +}(z)=z,
\end{equation}
for any $z\in C_{k, +}$
and
\begin{equation}
\label{eqn:arg}
	\arg (\varphi_{k, +} (z)-x_k)=\arg (z-x_k)
\end{equation}
for $z\in \widehat{C}_{k, 3}$.

Since
\begin{equation*}
	\log \frac{c-a}{c-b}=\log \left (1 + \frac{b-a}{c-b} \right )\leq \frac{b-a}{c-b}
\end{equation*}
for $0<a\leq b<c$, we obtain
\begin{equation}
\label{eqn:d+}
	d(\varphi_{k,+})\leq \frac{\log \widehat {r}_k -\log r_k}{\log \widetilde R_k  -\log \widehat r_k}.
\end{equation}
Moreover, we have
\begin{eqnarray}
\label{eqn:d+1}
	\log \widetilde R_k -\log \widehat r_k &=&\log \frac{(1+\delta)-(1-\delta)\widetilde q_{k+1}}{(1+\delta)-(1+\delta)\widetilde q_{k+1}} \\
	&\geq &\log \frac{(1+\delta)-(1-\delta)
	\delta}{(1+\delta)-(1+\delta)\delta } >0, \nonumber
\end{eqnarray}
and
\begin{eqnarray}
\label{eqn:d+2}
	\log \widehat r_k - \log r_k 
	=\log \frac{1-\widetilde q_{k+1}}{1-q_{k+1}}+\log \frac{(1+\delta)+(1-\delta)q_{k+1}}{(1+\delta)+(1-\delta)\widetilde q_{k+1}}.
\end{eqnarray}

We also see that
\begin{eqnarray}
\label{eqn:d+3}
	&\log \frac{(1+\delta)+(1-\delta)q_{k+1}}{(1+\delta)+(1-\delta)\widetilde q_{k+1}} \\
	&=\log \left \{ 1+ \frac{(1-\delta)(q_{k+1}-\widetilde q_{k+1})}{(1+\delta)+(1-\delta)\widetilde q_{k+1}} \right \}\nonumber 
	\leq  \frac{(1-\delta)(q_{k+1}-\widetilde q_{k+1})}{(1+\delta)+(1-\delta)\widetilde q_{k+1}}\\
	&\leq q_{k+1}-\widetilde q_{k+1}, \nonumber
\end{eqnarray}
because
\begin{equation*}
	\left (1+\delta\right )+\left (1-\delta\right )\widetilde q_{k+1}> 1-\delta >0.
\end{equation*}
From (\ref{eqn:d+})--(\ref{eqn:d+3}), we obtain
\begin{eqnarray}
\label{eqn:qcestimate1}
	&d(\varphi_{k, +})\leq \left (\log \frac{(1+\delta)-(1-\delta)
	\delta}{(1+\delta)-(1+\delta)\delta } \right )^{-1} \\
	&\times \left \{\log \frac{1-\widetilde{q}_{k+1}}{1-q_{k+1}}+(q_{k+1}-\widetilde{q}_{k+1}) \right \}\leq C(\delta)d(\omega, \widetilde \omega) 
	\nonumber
\end{eqnarray}
for some constant $C(\delta)>0$ depending only on $\delta$.

We may do the same operation, symmetrically; 
we take a circle $C_{k, -}$ centered at $-x_k$ of radius $\widetilde R_k$ and consider two annuli $A_{k, -}$ and $A_{k, -}'$.
The annulus $A_{k, -}$ is bounded by $C_{k, -}$ and $\widehat C_{k, 2}$, and $A_{k, -}'$ is bounded by $C_{k, -}$ and $C_{k, 2}$.
Then, we obtain a quasiconformal mapping $\varphi_{k, -} : A_{k, -}\to A_{k, -}'$ such that
\begin{equation}
\label{eqn:bdybwehavior1}
	\varphi_{k, -}(z)=z
\end{equation}
for $z\in C_{k, -}$ and
\begin{equation}
\label{eqn:arg1}
	\arg (\varphi_{k, -}(z)+x_k)=\arg (z+x_k).
\end{equation}
for $z\in \widehat{C}_{k, 2}$. Moreover, the mapping satisfies an inequality,
\begin{equation}
	d(\varphi_{k, -})\leq C(\delta)d(\omega, \widetilde \omega).
\end{equation}

We define a homeomorphism $\varphi_k : \widehat{P}_k\to P_k'$ by
\begin{equation*}
	\varphi_k (z)=
	\begin{cases}
		\varphi_{k, +}(z), \quad z\in A_{k, +} \\
		\varphi _{k, -}(z), \quad z\in A_{k, -} \\
		z, \quad \textrm{otherwise}.
	\end{cases}
\end{equation*}
The homeomorhpism $\varphi_k$ is quasiconformal except circles $C_{k, +}, C_{k, -}$. Hence, it has to be quasiconformal on $\widehat P_k$ with
\begin{equation}
	d(\varphi_k)\leq C(\delta)d(\omega, \widetilde \omega).
\end{equation}

\medskip
\noindent
{\bf Step 6 : Making quasiconformal mappings, II.}
In this step, we make a quasiconformal mapping from $P_k'$ to $\mathcal P_k$.
Recall that $P_k'$ is a pair of pants bounded by $\widehat C_{k, 1}$, $C_{k, 2}$ and $C_{k, 3}$, and $\mathcal P_k$ is bounded by $C_{k, 1}$, $C_{k, 2}$ and $C_{k, 3}$.

Let $C_{k, 0}$ be a circle centered at the origin of radius $x_k+r_k$, so that $C_{k, 0}$ is tangent with $C_{k, 2}$, $C_{k, 3}$. We consider circular annuli $B_k'$ bounded by $C_{k, 0}$ and $\widehat C_{k, 1}$, and $B_k$ bounded by $C_{k, 0}$ and $C_{k, 1}$.
It follows from Lemma \ref{lemma:modulus} that there exists a quasiconformal mapping $\psi_{k, 0} : B_k'\to B_k$ such that
\begin{equation*}
	d(\psi_{k, 0})=\log \frac{\log (1+\delta)\frac{x_k}{\widetilde x_k}-\log (x_k+r_k)}{\log (1+\delta)-\log (x_k+r_k)}
\end{equation*}
and $\psi_{k, 0}|_{C_0}$ is the identity.

As in Step 5, we have
\begin{equation*}
	d(\psi_{k, 0})\leq \frac{\log x_k-\log \widetilde x_k}{\log (1+\delta)-\log (x_k+r_k)}.
\end{equation*}

Now, we see that
\begin{eqnarray}
\label{eqn:d_}
	\log \left (1+\delta\right )-\log (x_k+r_k)&=&\log\frac{1+{\delta
	}}{1+\delta(1-q_{k+1})} \\
	&\geq &\log\frac{1+\delta }{1+{\delta}(1-\delta)}> 0, \nonumber
\end{eqnarray}
and
\begin{eqnarray}
\label{eqn:d_1}
	\log x_k -\log \widetilde x_k&=&\log\left (1+(1-\delta)\frac{q_{k+1}-\widetilde{q}_{k+1}}{(1+\delta)+(1-\delta)\widetilde{q}_{k+1}}\right ) \\
	&\leq & \left (1-\delta\right )\frac{q_{k+1}-\widetilde{q}_{k+1}}{(1+\delta)+(1-\delta)\widetilde{q}_{k+1}} \nonumber \\
	&\leq & q_{k+1}-\widetilde{q}_{k+1}. \nonumber
\end{eqnarray}
From (\ref{eqn:d_}) and (\ref{eqn:d_1}), we have
\begin{equation}
\label{eqn:qcestimate2}
	d(\psi_{k, 0})\leq \left ( \log\frac{1+\delta }{1+{\delta}(1-\delta)} \right )^{-1}(q_{k+1}-\widetilde{q}_{k+1}).
\end{equation}

We define a homeomorphism $\psi_k : P_k'\to \mathcal{P}_k$ by
\begin{equation*}
	\psi_k (z)=
	\begin{cases}
		\psi_{k, 0}(z), \quad z\in B_k' \\
		z, \quad \textrm{otherwise}.
	\end{cases}
\end{equation*}
Then, as in Step 5, we see that $\psi_k$ is quasiconformal on $P_k'$ with
\begin{equation}
	d(\psi_k)\leq C(\delta)d(\omega, \widetilde \omega).
\end{equation}

In the case where $q_{k+1}\leq \widetilde{q}_{k+1}$, the same argument still works in Steps 5 and 6; we obtain the same results.

\medskip
\noindent
{\bf Step 7 : Making a global quasiconformal mapping.}
In Steps 5 and 6, we have made quasiconformal mappings $\varphi_k :\widehat P_k\to P_k'$ and $\psi_k : P_k'\to \mathcal P_k$.
Thus, $\Phi_k:= \psi_k\circ\varphi_k : \widehat P_k\to \mathcal P_k$ gives a quasiconformal mapping with
\begin{equation*}
	d(\Phi_k)\leq C(\delta)d(\omega, \widetilde \omega)
\end{equation*}
for each $k\in \mathbb N$.

Because of the boundary behaviors (\ref{eqn:bdybehavior}), (\ref{eqn:arg}), (\ref{eqn:bdybwehavior1}) and (\ref{eqn:arg1}), we see that those mappings give a quasiconformal mapping $\Phi$ from $X_{E(\omega)}$ onto $X_{E(\widetilde{\omega})}$ with
\begin{equation*}
	d(\Phi)\leq C(\delta)d(\omega, \widetilde \omega).
\end{equation*}
Furthermore, from our construction of the mapping, we see that $\Phi (\mathbb H)=\mathbb H$.
Therefore, $\Phi$ is extended to a quasiconformal self-mapping of $\widehat{\mathbb C}$ as desired.
\qed

\medskip
{\it Proof of (2).} Take any $\varepsilon >0$.
Since, $\log \frac{1-\widetilde{q}_n}{1-q_n}\to 0$ as $n\to \infty$, we also see that $q_n\to \widetilde{q}_n\to 0$.
Viewing (\ref{eqn:qcestimate1}) and (\ref{eqn:qcestimate2}), we verify that there exists an $N\in \mathbb N$ such that 
\begin{eqnarray*}
	d(\varphi_k)<\frac{1}{2}\log (1+\varepsilon) \quad \textrm{and}\quad	d(\psi_k)<\frac{1}{2}\log (1+\varepsilon),
\end{eqnarray*}
if $k>N$. Hence, if $k>N$, then
\begin{equation}
\label{eqn:asymtotic}
	d(\Phi_k)=d(\psi_k\circ\phi_k)\leq d(\psi_k)+d(\varphi_k)<\log (1+\varepsilon).
\end{equation}

Since the pants decompositions in Step 2 of the proof (1) give  exhaustions $X_{E(\omega)}$ and $X_{E(\widetilde \omega)}$, (\ref{eqn:asymtotic}) implies the maximal dilatation $K(\Phi)=e^{d(\Phi)}$ is less than $(1+\varepsilon)$ on the outside of a sufficiently large compact subset of $X_{E(\omega)}$.
Therefore, $\Phi : X_{E(\omega)}\to X_{E(\widetilde\omega)}$ is asymptotically conformal.
\qed

\section{Proof of Theorem \ref{Them:Main3}}
Suppose that there exists a $K$-quasiconformal map from $X_{\mathcal C}$ to $X_{E(\omega)}$.
	Let $d>0$ be the smallest hyperbolic length in all simple closed curves in $X_{\mathcal C}$.
	By Wolpert's formula (cf. \cite{ShigaHyper}, \cite{Wolpert}), the hyperbolic length of any simple closed curve in $X_{E(\omega)}$ is not less than $K^{-1}d$.
	
	Let $\varepsilon >0$ be an arbitrary small constant.
	Since $\sup \{q_n\mid n\in \mathbb N\}=1$, there exist a sequence $\{n_k\}_{k=1}^{\infty}$ in $\mathbb N$ and $N_0\in \mathbb N$ such that 
	\begin{equation*}
		1-\varepsilon < q_{n_k} <1,
	\end{equation*}
	if $k>N_0$.
	
	Now we look at $I_{q_{k}-1}^{1}$ of $E_{q_{k}-1}$ for $k>N_0$.
	Then, $I_{q_k}^{1}\subset E_{q_{k}}$ is an interval of length $\frac{1}{2}(1-q_k)|I_{q_{k}-1}^{1}|<\frac{1}{2}\varepsilon |I_{q_{k}-1}^{1}|$.
	Therefore, we may take an annulus $A_{k}$ in $X_{E(\omega)}$ bounded by two concentrated circles $C_{k}^{1}, C_{k}^{2}$ such that the radius of $C_{k}^{1}$ is $\frac{1}{4}\varepsilon |I_{q_{k}-1}^{1}|$ and that of $C_{k}^{2}$ is $(\frac{1}{2}-\frac{1}{4}\varepsilon) |I_{q_{k}-1}^{1}|$.
	If we take $\varepsilon >0$ sufficiently small, then the length of the core curve of $A_k$ with respect to the hyperbolic metric on $A_k$ becomes smaller than $K^{-1}d$.
	Since $A_k\subset X_{E(\omega)}$, the length of the core curve of $A_k$ with respect to the hyperbolic metric of $X_{E(\omega)}$ is not greater than the length with respect to the hyperbolic metric of $A_k$.
	Thus, we find a closed curve in $X_{E(\omega)}$ whose length is less that $K^{-1}d$.
It is a contradiction and we complete the proof of the theorem.
	
\section{Proofs of Corollaries}
\noindent
{\bf{Proof of Corollary \ref{Cor1}.}}
Let $\Lambda$ be the limit set of the Schottky group $G$. 
We have shown (\cite{ShigaPrepri} Theorem 6.2) that $X_{\Lambda}$ is quasiconformally equivalent to $X_{\mathcal C}$.
Hence, it follows from Theorem \ref{Thm:Main} that $X_{E}$ is quasiconformally equivalent to $X_{\Lambda}$ as desired.
\qed

\medskip
\noindent
{\bf Proof of Corollary \ref{Cor3}.} Since $\mathcal C=E(\omega_0)$ for $\omega_0=(\frac{1}{3})_{n=1}^{\infty}$, the statement follows immediately from Theorem \ref{Thm:Main2} (1). \qed

\medskip
\noindent
{\bf{Proof of Corollary \ref{Cor2}}.}
Let $\varphi : X_{\Lambda}\to X_{E}$ be a quasiconformal map given by Corollary \ref{Cor1}. 
Take any quasiconformal map $\psi$ on $X_{E}$ to $\widehat{\mathbb C}$.
Then, $\Phi :=\psi\circ\varphi$ be a quasiconformal map on $X_{\Lambda}$.
It is known that any quasiconformal map on $X_{\Lambda}$ is extended to a quasiconformal map on $\widehat{\mathbb C}$ (cf. 
\cite{ShigaKlein}).
Hence, both $\varphi$ and $\Phi$ are extended to $\widehat{\mathbb C}$ and so is $\psi=\Phi\circ\varphi^{-1}$.
\qed

\medskip
\noindent
{\bf Proof of Corollary \ref{Cor:Hausdorff}.}
Let $\Psi : \mathbb C\to \mathbb C$ be the quasiconformal mapping given in \S 4.
We put $D=\textrm{dim}_H (E(\omega))$ and $\widetilde{D}=\textrm{dim}_{H}(E(\widetilde{\omega}))$.
We use the argument in the proof of Theorem \ref{Thm:Main2} (2).

For any $\varepsilon >0$, there exists $N\in \mathbb N$ such that
\begin{equation*}
	K(\Phi_{k})<1+\varepsilon
\end{equation*}
if $k>N$, where $\Phi_k$ is the quasiconformal mapping given in \S 4.
Therefore, $\Phi|_{U_N}$ is a $(1+\varepsilon)$-quasiconformal mapping on $U_N:=E(\omega)\cup\bigcup_{k>N}\bigcup_{i=1}^{2^k}P_k^i$.
Here, we use the following result by Astala \cite{Astala}.

\begin{Pro}
	Let $\Omega, \Omega'$ be planar domains and $f : \Omega\to \Omega'$ $K$-quasiconformal mapping.
	Suppose that $E\subset \Omega$ is a compact subset of $\Omega$.
	Then,
	\begin{equation}
	\label{eqn:Astala}
		\textrm{dim}_{H}(f(E))\leq \frac{2K\textrm{dim}_{H}(E)}{2+(K-1)\textrm{dim}_{H}(E)}.
	\end{equation}
\end{Pro}
It follows from (\ref{eqn:Astala}) that
\begin{equation*}
	\textrm{dim}_{H}(E(\widetilde{\omega}))\leq \frac{2(1-\varepsilon)\textrm{dim}_{H}(E(\omega))}{2+\varepsilon\textrm{dim}_{H}(E(\omega))}.
\end{equation*}
Since $\varepsilon >0$ could be an arbitrary small, we obtain
\begin{equation*}
	\textrm{dim}_{H}(E(\widetilde{\omega}))\leq \textrm{dim}_{H}(E(\omega)).
\end{equation*}

By considering $\Phi^{-1}$, we get the reverse inequality for $\textrm{dim}_{H}(E(\omega))$ and $\textrm{dim}_{H}(E(\widetilde{\omega}))$. Thus, we conclude that $\textrm{dim}_{H}(E(\omega))=\textrm{dim}_{H}(E(\widetilde{\omega}))$ as desired.
\qed
\section{Examples} 
	\begin{Ex}
	\label{Ex:hyperRat}
Let $f_c (z)=z^2+c$.
Suppose that $c$ is not in the Mandelbrot set.
Then, it is well known that $f_c$ is hyperbolic and the Julia set $\mathcal J_{f_c}$ is a Cantor set.
Thus, $f_c$ satisfies the condition of Theorem \ref{Thm:Main}.
\end{Ex}
\begin{Ex}
\label{Ex:hypBl}
Let $B_0(z)$ be a Blaschke product of degree $d>1$.
Suppose that $B_0$ has an attracting fixed point on the unit circle $T:=\{|z|=1\}$.
Since the Julia set $\mathcal J_{B_0}$ of $B_0$ is included in $T$, it has to be a Cantor set.
It is also easy to see that $B_0$ is hyperbolic.
Thus, $B_0$ satisfies the condition on Theorem I.
\end{Ex}

In Theorem \ref{Thm:Main2}, we have estimated the maximal dilatations for sequences with lower bound. 
In next example, we may estimate the maximal dilatation for sequences without lower bound. 
\begin{Ex}
	For $0<a<1$ and a fixed $L\in \mathbb N$, we put $q_n=a^{n}$ and $\widetilde{q}_n=a^{n+L}$ and we consider $E(\omega)$, $E(\widetilde{\omega})$ for $\omega=(q_n)_{n=1}^{\infty}$, $\widetilde{\omega}=(\widetilde{q}_n)_{n=1}^{\infty}$.
	By using the same idea as in the proof of Theorem \ref{Thm:Main2}, we claim that there exists an $\exp (Ca^{-L})$-quasiconformal mapping $\varphi : \mathbb C\to \mathbb C$ with $\varphi (E(\omega))=E(\widetilde{\omega})$, where $C>0$ is a constant independent of $\omega$ and $\widetilde{\omega}$.
	
	\medskip
	\noindent
	{\bf Proof of the claim.} We use the same notations for $E(\omega)$ and $E(\widetilde{\omega})$ as those in the proof of Theorem \ref{Thm:Main2}.
	Then, 
	$$
	E_k=\cup_{i=1}^{2^k}I_k^i,\quad [0, 1]=E_k\cup\bigcup_{i=1}^{2^k-1}J_k^i
	$$ 
	and for $i=1, 2, \dots , 2^k$,
	$$
	|I_{k}^{i}|=\left (\frac{1}{2}\right )^{k}\prod_{j=1}^{k}(1-a^j). 
	$$
	If $i$ is odd, then
	\begin{equation*}
		|J_{k+1}^i|=a^{k+1}|I_k^1|\geq 2a^{k+1}|I_{k+1}^1|.
	\end{equation*}
	If $i=2^{\ell}m$ $(1\leq\ell\leq k; m$ is odd), then we have
	\begin{equation*}
		|J_{k+1}^i|=|J_{k-\ell +1}^m|\geq 4a^{k+1}|I_{k+1}^1|.
	\end{equation*}
	Thus, we conclude that
	\begin{equation}
	\label{eqn:NoLB}
		|J_{k+1}^i|\geq 2a^{k+1}|I_{k+1}^1|,
	\end{equation}
	for $i=1, 2, \dots 2^{k+1}-1$.
	
	We draw a circle $C_k^i$ centered at the midpoint of $I_{k}^i$ with radius $\frac{1}{2}(1+a^k)|I_{k}^1|$ for each $k\in \mathbb N$ and $1\leq i\leq 2^k$.
From (\ref{eqn:NoLB}), we see that $C_k^i\cap C_k^j=\emptyset$ if $i\not=j$.
Therefore, $\cup_{k=1}^{\infty}\cup_{i=1}^{2^k}C_k^i$ gives a pants decomposition of $X_{E(\omega)}$.
We also draw circles $\widetilde{C}_{k}^i$ for $\widetilde{\omega}$ by the same way.
Then, $\cup_{k=1}^{\infty}\cup_{i=1}^{2^k}\widetilde{C}_k^i$ gives a pants decomposition of $X_{E(\widetilde{\omega})}$.

We denote by $P_k^i$ a pair of pants bounded by $C_k^i, C_{k+1}^{2i-1}$ and $C_{k+1}^{2i}$.
As in Step 3 of the proof of Theorem \ref{Thm:Main2}, we may identify $P_{k}^i$ with a pair of pants $\mathcal{P}_k$ bounded by $C_{k, 1}, C_{k, 2}$ and $C_{k, 3}$, where $C_{k, 1}$ is a circle centered at the origin with radius $1+a^k$, $C_{k, 2}$ is centered at
\begin{equation*}
	-x_k :=-a^{k+1}-\frac{1}{2}(1+a^{k+1})(1-a^{k+1})
\end{equation*}
with radius
\begin{equation*}
	r_k :=\frac{1}{2} (1+ a^{k+1})(1-a^{k+1})
\end{equation*}
and $C_{k, 3}$ is centered at $x_k$ with radius $r_k$.

Similarly, we take a pair of pants $\widetilde{P}_k^i$ bounded by $\widetilde C_k^i, \widetilde  C_{k+1}^{2i-1}$ and $\widetilde C_{k+1}^{2i}$, which is conformally equivalent to a pair of pants $\widetilde{\mathcal P}_k$ bounded by $\widetilde C_{k, 1}, \widetilde C_{k, 2}$ and $\widetilde C_{k, 3}$, where 
$\widetilde C_{k,1}$ is the same circle as $C_{k,1}$, $\widetilde C_{k,2}$ is centered at
\begin{equation*}
	-\widetilde x_k :=--a^{k+L+1}-\frac{1}{2}(1+a^{k+L+1})(1-a^{k+L+1})
\end{equation*}
wirh radius
\begin{equation*}
	\widetilde r_k :=\frac{1}{2} (1+ a^{k+L+1})(1-a^{k+L+1})
\end{equation*}
and $\widetilde C_{k, 3}$ is centered at $\widetilde{x}_k$ with radius $\widetilde{r}_k$.

We also take an intermediate pair of pants, $\widehat{P}_k$ similar to that of the proof of Theorem \ref{Thm:Main2}.
Then, by using exactly the same method, we may construct a $\exp (Ca^{-L})$-quasiconformal mapping from $P_{k}^i$ onto $\widetilde{P}_k^i$, where $C>0$ is a constant independent of $k$ and $i$. Since the calculation is a bit long but the same as in \S 4, we may leave it to the reader.

By gluing those quasiconformal mappings together, we get an $\exp (Ca^{-L})$-quasiconformal mapping $\varphi : \mathbb C\to \mathbb C$ with $\varphi (E(\omega))=E(\widetilde{\omega})$ as desired. \qed
\end{Ex}  

\medskip
\noindent
{\bf Cantor Julia sets of Blaschke products with parabolic fixed points.}

We showed (\cite{ShigaPrepri} Example 3.2) that a Cantor set which is the limit set of an extended Schottky group is not quasiconformally equivalent to the limit set of  a Schottky group. 
We discuss the same thing for Cantor sets defined by non-hyperbolic rational functions.

Let $B_1(z)$	be a Blaschke product with a parabolic fixed point on the unit circle $T$.
Suppose that there exists only one attracting petal at the parabolic fixed point.
Then, we see that the Julia set $\mathcal J_{B_1}$ is a Cantor set on $T$ (see \cite{Carleson-Gamelin} IV. 2. Example).
However, $B_1$ is not hyperbolic since it has a parabolic fixed point.

It follows from Theorem \ref{Thm:Main} that  two Riemann surfaces $X_{\mathcal J_{f_c}}$ for Example \ref{Ex:hyperRat} and $X_{\mathcal J_{B_0}}$ for Example \ref{Ex:hypBl} are quasiconformally equivalent.
While the Julia set $\mathcal J_{B_1}$ of $B_1$ is also a Cantor set, it is not hyperbolic.
Therefore, we cannot apply Theorem \ref{Thm:Main} for $B_1$.

	Now, we consider the Martin compactification of the complement. For a general theory of the Martin compactification, we may refer to \cite{Constantinescu-Cornea}.
	Here, we note the following.
	\begin{Pro}
	\label{HyperBl}
		Let $B$ be a hyperbolic Blaschke product of degree $d>1$. 
		Suppose that the Julia set $\mathcal J_{B}$ is a Cantor set in $T$.
		Then,
		 the Martin compactification of $X_{\mathcal J_{B}}$ is homeomorphic to $\widehat{\mathbb C}$. 
	\end{Pro}
	
	Hence, the same statements as in Proposition \ref{HyperBl} hold for $X_{\mathcal J_{0}}:=\widehat{\mathbb C}\setminus\mathcal J_0$ and the quasiconformal map $\varphi$ on $X_{\mathcal J_{0}}$ is extended to a homeomorphism of the Martin compactification of $X_{\mathcal J_{0}}$.
	
	Next, we consider the Martin compactification of $X_{\mathcal J_{1}}$, especially the set of the Martin boundary over the parabolic fixed point of $B_1$.
	If the set contains at least two points, then it follows from Proposition \ref{HyperBl} that there exists no quasiconformal map from $X_{\mathcal J_0}$ to $X_{\mathcal J_1}$.
	
	Indeed, in \cite{ShigaKlein} we observe the 
	Martin compactification of the complement of the limit set of an extended Schottky group and show that the set of the Martin boundary over a parabolic fixed point consists of more than two points. It is a key fact to show that the limit set of the extended Schottky group is not quasiconformally equivalent to that of a Schottky group (\cite{ShigaPrepri}).
	However, by using an argument of Benedicks (\cite{Benedicks}) (see also Segawa \cite{Segawa}) on the Martin compactification, we may show the following.
	\begin{lemma}
	\label{parabBl}
		In the Martin compactification of $X_{\mathcal J_{1}}$, there is exactly one minimal point over the parabolic fixed point of $B_1$.
	\end{lemma}
\begin{Rem}
	In the Martin compactification of a Riemann surface, the set corresponding to a topological boundary component of the Riemann surface is connected and the minimal points in the set are regarded as extreme points of a convex set. Thus, if the set over a boundary component on the Martin compactification contains only one minimal point, then it consists of only one point, that is, the minimal point.
\end{Rem}	
	\begin{proof}
	To prove the lemma, we use a result by Benedicks.
	
	We denote by $Q(t, r)$ $(t\in \mathbb R, r>0)$, the square
	\begin{equation*}
		\left \{x+iy\mid |x-t|<\frac{r}{2}, |y|<\frac{r}{2}\right \}.
	\end{equation*} 
	For a fixed $\alpha$ with $0<\alpha <1$ and every $x\in \mathbb R\setminus\{0\}$, we consider the solution of the Dirichlet problem on $Q(x, \alpha |x|)\setminus E$ with boundary values one on $\partial Q(x, \alpha |x|)$ and zero on $E\cap Q(x, \alpha |x|)$. 
	We denote the solution by $\beta_{x}^{E}$.
	Then, Benedicks showed the following.
	\begin{Pro}
	\label{Pro:Benedicks}
		On the Martin compactification of $\widehat{\mathbb C}\setminus E$, there exist more than two points over $\infty$ if and only if 
		\begin{equation}
			\int_{|x|\geq 1}\frac{\beta_{x}^{E}(x)}{|x|}dx <\infty.
		\end{equation}
	\end{Pro}

		Let $a\in T$ be the parabolic fixed point $B_1$.
	We take a M\"obius transformation $\gamma$ so that $\gamma (T)=\mathbb R\cup \{\infty\}$ and $\gamma (a)=\infty$.
For $\widehat{B}_1:=\gamma B_1\gamma^{-1}$, we see that $\infty$ is a parabolic fixed point with a unique attracting petal of $\widehat{B}_1$, and $\mathcal J_{1}:=\gamma (\mathcal J_{B_1})$ is contained in $\mathbb R\cup\{\infty\}$.
	
		Since $z=\infty$ is a parabolic fixed point of $\widehat{B}_1$ with only one attracting pegtal, we may assume that there exists a sufficiently large $M>0$ such that $\mathcal J_{1}\cap\{\textrm{Re }z<-M\}$ is empty while $\mathcal J_{1}\cap\{\textrm{Re }z>M\}$ is not empty.
		Hence, $\mathcal J_{1}\cap Q(x, \alpha |x|)=\emptyset$ if $x<0$ and $|x|$ is sufficiently large.
		Therefore, $\beta_{x}^{\mathcal J_1}(x)=1$ for such $x$.
		Thus, we have
		\begin{equation*}
			\int_{|x|\geq 1}\frac{\beta_{x}^{\mathcal J_1}(x)}{|x|}dx=\infty	
		\end{equation*}
		and conclude that there exists exactly one point over $\infty$ from Proposition \ref{Pro:Benedicks}.
\end{proof}
Lemma \ref{parabBl} implies that we cannot use the argument used for extended Schottky groups.
We exhibit the following conjecture at the end of this article.

{\it Conjecture.} $X_{\mathcal J_1}$ is not quasiconformally equivalent to $X_{\mathcal C}$.

\end{document}